\documentclass[twoside,9pt]{article}
\usepackage{graphicx,amsmath,amsthm,amssymb,latexsym,amsfonts,color}
\usepackage[bookmarksnumbered, colorlinks, plainpages]{hyperref}
\newtheorem{thm}{\bf Theorem}
\newtheorem{lem}{\bf Lemma}

\newtheorem{remark}{\bf Remark}

\newtheorem{alg}{Algorithm}
\newcommand{\norm}[1]{{ \left\| { #1 } \right\| }}
\def\h{\hspace{-0.2cm}}
\begin{document}
\title{ On the generalized shift-splitting preconditioner for saddle point problems}
\author{{ Davod Khojasteh Salkuyeh\footnote{Corresponding author}, Mohsen Masoudi and Davod Hezari}\\[2mm]
\textit{{\small Faculty of Mathematical Sciences, University of Guilan, Rasht, Iran}} \\
\textit{{\small E-mails: khojasteh@guilan.ac.ir, masoodi\_mohsen@ymail.com, d.hezari@gmail.com}\textit{}}}
\date{}
\maketitle
{\bf Abstract.} In this paper, the generalized shift-splitting preconditioner is implemented  for saddle point problems with symmetric positive definite (1,1)-block and symmetric positive semidefinite (2,2)-block. The proposed preconditioner is extracted form a stationary iterative method which is unconditionally convergent.  Moreover, a relaxed version of the proposed preconditioner is presented and some properties of the  eigenvalues distribution  of the  corresponding preconditioned matrix are studied. Finally, some numerical experiments on test problems arisen from  finite element discretization of  the Stokes problem  are given to show the effectiveness of the  preconditioners.
 \\[-3mm]


\pagestyle{myheadings}\markboth{D.K. Salkuyeh, M. Masoudi and D. Hezari}{Generalized shift-splitting preconditioner for saddle point problems}

\thispagestyle{empty}

\section{Introduction} \label{Sec1}

Consider the saddle point linear system
\begin{equation}\label{sdpr1}
\mathcal{A}u\equiv\left(
  \begin{array}{cc}
    A & B^T \\
   -B & C
  \end{array}
\right)
\left(
  \begin{array}{c}
    x \\
    y
  \end{array}
\right)=
\left(
  \begin{array}{c}
    f \\
    -g
  \end{array}
\right)\equiv b,
\end{equation}
where $A\in \mathbb{R}^{n \times n}$ is symmetric positive definite (SPD), $C\in \mathbb{R}^{m \times m}$ is symmetric positive semidefinite and $B\in \mathbb{R}^{m\times n}$, $m\leq n$, is of full rank. Moreover, $x,f\in \mathbb{R}^{n}$ and $y,g\in \mathbb{R}^{m}$. We also assume that the matrices $A$, $B$ and $C$ are large and sparse. According to Lemma 1.1 in \cite{Benzi-a1} the matrix $\mathcal{A}$ is nonsingular.  Such systems arise in a variety of scientific computing and engineering applications, including constrained optimization, computational fluid dynamics, mixed finite element discretization of the Navier-Stokes equations, etc. (see \cite{Axel,Benzi-a,Brezzi,Elman}). Application-based analysis can be seen in \cite{Elman1,Klawonn,Silvester}.

In the last decade, there has been intensive work on development of the effective iterative methods for solving matrix equations with different structures (see for example \cite{Bai-HSS,ding1,ding2,ding3,ding4,Saadbook}).  Benzi and Golub \cite{Benzi-a1} investigated the convergence and the preconditioning properties of the Hermitian and skew-Hermitian splitting (HSS) iterative method \cite{Bai-HSS}, when it is used for solving the saddle point problems. Bai et al. in \cite{Bai-PHSS} established the preconditioned HSS (PHSS) iterative method, which involves a single parameter, and then, Bai and Golub in \cite{Bai-AHSS} proposed its two-parameter acceleration, called the accelerated Hermitian and skew-Hermitian splitting (AHSS) iterative method; see also \cite{Bai-HSS-like}. Besides these HSS methods, Uzawa-type schemes \cite{Uzawa1,Uzawa2,Uzawa3,Uzawa4,Uzawa6} and preconditioned Krylov subspace methods, such as MINRES and GMRES incorporated with suitable preconditioners  have also been applied to solve  the saddle point problems (see \cite{wu1,wu2,wu3,wu4} and the references therein as well as \cite{prec3,prec4}). The reader is also referred to  \cite{Benzi-a} for a comprehensive survey.

To solve the  saddle point problem \eqref{sdpr1} when $C=0$, Cao et al., in \cite{SSP}, proposed the  shift-splitting  preconditioner
\[
\mathcal{P}_{SS}=\frac{1}{2}(\alpha I+ \mathcal{A})=\frac{1}{2}
\left(
  \begin{array}{cc}
    \alpha I+A & B^T \\
           -B & \alpha I \\
  \end{array}
\right),
\]
which is a skillful generalization of the idea of the shift-splitting  preconditioner  initially introduced in \cite{Bai-PSS} for solving a non-Hermitian positive definite linear system where $\alpha>0$ and $I$ is the identity matrix.
Recently, Chen and Ma in \cite{GSS} studied   the  two-parameter generalization of the preconditioner $\mathcal{P}_{SS}$, say
\[
\mathcal{P}_{GSS}=\frac{1}{2}
\left(
  \begin{array}{cc}
   \alpha I+A & B^T \\
  -B & \beta I \\
  \end{array}
\right),
\]
for solving the saddle point linear systems \eqref{sdpr1} with $C=0$, where $\alpha\geq 0$ and $\beta>0$.


In this paper, we propose a modified generalized shift-splitting (MGSS) preconditioner for the saddle point problem \eqref{sdpr1} with $C\neq0$. The MGSS preconditioner is based on a splitting of the saddle point matrix which  results in an unconditionally convergent stationary iterative method. Moreover, a relaxed version of the MGSS preconditioner is presented and  the eigenvalues distribution of the corresponding preconditioned matrix is studied.

The organization of the paper is as follows. In Section \ref{Sec2} we propose the MGSS preconditioner and its relaxed version. Section \ref{Sec3} is devoted to some numerical experiments. Finally, in Section \ref{Sec4} we present some concluding remarks.


\section{The generalized shift-splitting preconditioner}\label{Sec2}

Let $\alpha , \beta >0$. Consider the splitting  $\mathcal{A}=\mathcal{M}_{\alpha,\beta}-\mathcal{N}_{\alpha,\beta}$, where
\begin{equation}\label{MGHSSsplit}
\mathcal{M}_{\alpha,\beta}=
\frac{1}{2}
\left( {\begin{array}{*{20}{c}}
\alpha I +A & B^T\\
-B & \beta I+ C
\end{array}} \right) \quad \textrm{and} \quad
\mathcal{N}_{\alpha,\beta}=\frac{1}{2}
\left( {\begin{array}{*{20}{c}}
\alpha I -A & -B^T\\
B & \beta I- C
\end{array}} \right).
\end{equation}
This splitting leads to the following stationary iterative method (the MGSS iterative scheme)  \eqref{sdpr1}
\begin{equation}\label{MGHSSit}
\mathcal{M}_{\alpha,\beta} u^{(k+1)}=\mathcal{N}_{\alpha,\beta} u^{(k)}+b
\end{equation}
for solving the linear system \eqref{sdpr1}, where $u^{(0)}$ is an initial guess.
Therefore, the iteration matrix of the MGSS iterative method is given by $\Gamma_{\alpha,\beta} =\mathcal{M}_{\alpha,\beta}^{-1}\mathcal{N}_{\alpha,\beta}$.
In the sequel, the convergence of the proposed method is studied. It is well known that the iterative method \eqref{MGHSSit} is convergent for every initial guess $u^{(0)}$ if and only if $\rho(\Gamma_{\alpha,\beta})<1$, where $\rho(.)$ denotes the spectral radius of $\Gamma$ (see \cite{Axel}).
Let $u=(x ; y)$ be an eigenvector corresponding to the  eigenvalue $\lambda $ of $\Gamma_{\alpha,\beta}$. Then, we have
$\mathcal{N}_{\alpha,\beta}u=\lambda \mathcal{M}_{\alpha,\beta}u$ or equivalently
\begin{eqnarray}
  (\alpha I-A)x-B^Ty &\h=\h& \lambda (\alpha I+A)x+\lambda B^Ty, \label{EigEq1} \\
  Bx+(\beta I-C)y &\h=\h& -\lambda Bx+ \lambda (\beta I+C) y.  \label{EigEq2}
\end{eqnarray}


\begin{lem}\label{Lem1}
Let $\alpha,\beta>0$. If $\lambda$ is an eigenvalue of the matrix  $\Gamma_{\alpha,\beta}$, then $\lambda \neq \pm1$.
\end{lem}

\begin{proof}
If $\lambda=1$, then from Eqs. \eqref{EigEq1} and  \eqref{EigEq2} we obtain $\mathcal{A}u=0$ which is a contradiction, since $u\neq 0$ and the matrix $\mathcal{A}$ is nonsingular (\cite[Lemma 1.1]{Benzi-a1}).

If $\lambda=-1$, then from Eqs. \eqref{EigEq1} and  \eqref{EigEq2} it follows that $2\alpha x=0$ and $2\beta x=0$. Since, $\alpha,\beta>0$, we get $x=0$ and $y=0$. This is a contradiction, because $(x;y)$ is an eigenvector of $\mathcal{A}$.
\end{proof}


\begin{thm}
\label{th1}
Let $\lambda$ be an eigenvalue of the matrix $\Gamma$ and $\alpha , \beta >0$. Then $| {\lambda}|<1$.
\end{thm}
\begin{proof}
We first show that $x\neq 0$. If $x=0$, then it follows from Eq. \eqref{EigEq1} that $(1+\lambda ) B^T y=0$. Therefore, from Lemma \ref{Lem1} we conclude that $B^Ty=0$ and this yields $y=0$, since $B$ has full rank. This is a contradiction because $(x;y)$ is an eigenvector of $\Gamma_{\alpha,\beta}$.

Without loss of generality let $\norm{x}_2=1$. Multiplying both sides of \eqref{EigEq1} by $x^*$ yields
\begin{equation}\label{EigEq3}
\alpha-x^*Ax -(Bx)^*y=\lambda (\alpha \norm{x}_2^2+x^*Ax)+\lambda (Bx)^*y.
\end{equation}
We consider two cases $Bx=0$ and $Bx \neq 0$. If $Bx=0$, then Eq. \eqref{EigEq3} implies
\begin{eqnarray*}
|\lambda| =\frac{|\alpha  -x^*Ax|}{|\alpha +x^*Ax|}<1.
 \end{eqnarray*}
We now assume that $Bx \neq 0$. In this case, from Eq. \eqref{EigEq2} we obtain
\begin{equation}\label{EigEq4}
 Bx=\frac{\beta (\lambda-1)}{\lambda+1}y+Cy.
 \end{equation}
Substituting Eq. \eqref{EigEq4} in \eqref{EigEq3} yields
 \begin{equation*}
 (1-\lambda) \alpha -(1+\lambda)x^*Ax=(1+\lambda)\left( \beta \frac{ \overline{\lambda}-1}{1+\overline{\lambda}}y^*y+y^*Cy\right).
 \end{equation*}
Letting $p=x^*Ax$, $q=y^*y$, and $r=y^*Cy$, it follows from the latter equation that
\begin{equation}\label{Eqwwb}
\alpha \omega + \beta q \overline{\omega}=p+r,\quad \textrm{with}\quad \omega=\frac{1-\lambda}{1+\lambda}.
\end{equation}
Since $\alpha,\beta,p>0$ and $ q,r\geq 0$, form \eqref{Eqwwb} we see that
\[
\Re(w)=\frac{p+r}{\alpha+\beta q}>0.
\]
Hence, we have
\[
|\lambda|=\frac{|1-\omega|}{|1+\omega|}=\sqrt { \frac{ (1-\Re(\omega))^2+\Im(\omega)^2 }{ (1+\Re(\omega))^2+\Im(\omega)^2 } }<1
\]
which completes the proof.
\end{proof}


\begin{remark}
Let $C=0$. If $\alpha=\beta>0$ then MGSS iterative method is reduced to the shift-splitting method presented by Cao et al. in \cite{SSP} and when $\alpha$ and $\beta$ are two positive parameters the method becomes  the generalized shif-splitting method proposed by Chen and Ma in \cite{GSS}.
\end{remark}

Theorem \ref{th1} guarantees the convergence of the MGSS method, however  the stationary iterative method \eqref{MGHSSit} is typically too slow  for the method to be competitive. Nevertheless, it serves the preconditioner $\mathcal{P}_{MGSS}=\mathcal{M}_{\alpha,\beta}$ for a Krylov subspace method such as GMRES, or its restarted version GMRES($m$) to solve system \eqref{sdpr1}. At each step of the MGSS iterative method or applying the shift-splitting preconditioner $\mathcal{P}_{MGSS}$ within a Krylov subspace  method, we need to compute a vector of the form $z=\mathcal{P}_{MGSS}^{-1}r$ for a given $r=(r_1;r_2)$ where $r_1\in \Bbb{R}^n$ and $r_2\in \Bbb{R}^m$.  It is not difficult to check that
\begin{align*}
\mathcal{P}_{MGSS}&= \frac{1}{2} \left( {\begin{array}{*{20}{c}}
I &B^T \left(\beta I+C \right) ^{-1}\\
0 & I
\end{array}} \right)
\left( {\begin{array}{*{20}{c}}
S  & 0 \\
 0 & \beta I+C
\end{array}} \right)
\left( {\begin{array}{*{20}{c}}
I & 0 \\
- \left( \beta I+C \right)^{-1} B & I
\end{array}} \right),
\end{align*}
where $ S= \alpha I+ A+B^T \left( \beta I+ C \right)^{-1}B$. Hence,
\begin{align}\label{Pinv}
\mathcal{P}_{MGSS}^{-1}&= 2
 \left( {\begin{array}{*{20}{c}}
I & 0 \\
\left( \beta I+C \right)^{-1} B & I
\end{array}} \right)
\left( {\begin{array}{*{20}{c}}
S^{-1}  & 0 \\
 0 & ( \beta I+C)^{-1}
\end{array}} \right)
\left( {\begin{array}{*{20}{c}}
I & -B^T \left(\beta I+C \right) ^{-1}\\
0 & I
\end{array}} \right).
\end{align}
By using Eq. \eqref{Pinv} we state Algorithm  \ref{Alg1} to compute the vector $z=(z_1;z_2)$ where $z_1\in \Bbb{R}^n$ and $z_2\in \Bbb{R}^m$ as following.

\begin{alg}\rm\label{Alg1}
Computation of $z=\mathcal{P}_{MGSS}^{-1}r$. \\[-5mm]
\begin{enumerate}
\item Solve $(\beta I +C)w=2r_2$ for $w$. \\[-6mm]
\item Compute $w_1=2r_1-B^Tw$. \\[-6mm]
\item Solve $\left( \alpha I+ A+B^T(\beta I+C)^{-1}B\right)z_1=w_1$ for $z_1$. \\[-6mm]
\item Solve  $(\beta I +C)v=B z_1$ for $v$.\\[-6mm]
\item Compute $z_2=v+w$.
\end{enumerate}
\end{alg}

Obviously, the matrix $S=\alpha I+ A+B^T(\beta I+C)^{-1}B$ is SPD. In practical implementation of Algorithm \ref{Alg1} one may use the conjugate gradient (CG) method or a preconditioned CG (PCG) method to solve the system of  Step 3. It is noted that, since the matrix $\beta I+ C$ is SPD and of small size in comparison to the size of $A$, we use the Cholesky factorization of $\beta I+ C$ in Steps 1, 3 and 4.

In the sequel, we consider the relaxed  MGSS (RMGSS)  preconditioner
\[
\mathcal{P}_{RMGSS}=
\left( {\begin{array}{*{20}{c}}
A & B^T\\
-B & \beta I+ C
\end{array}} \right).
\]
for the saddle point  problem \eqref{sdpr1}. The next theorem discusses eigenvalues distribution of  $\mathcal{P}_{RMGSS}^{-1}\mathcal{A}$.

\begin{thm}\label{th2}
The preconditioned matrix $\Psi=\mathcal{P}_{RMGSS}^{-1}\mathcal{A}$ has an eigenvalue 1 with multiplicity $n$ and the remaining eigenvalues are
$\lambda_i= \frac{\mu_i}{\beta+\mu_i}$,  $1 \leq i \leq m$, where $\mu_i$'s are the  eigenvalues of the matrix $G=C+B A ^{-1} B^T$.
\end{thm}
\begin{proof}
By using Eq. \eqref{Pinv} (with $\alpha=0$ and neglecting the pre-factor 2) we obtain
\begin{align*}
\Psi=
\left( {\begin{array}{*{20}{c}}
I & A^{-1}B^T-A^{-1}B^T S^{-1}BA^{-1}B^{T}-A^{-1} B^T S^{-1}C \\
0 & S^{-1} G
\end{array}} \right)
\end{align*}
where $S=\beta I +G$. Therefore, $\Psi$  has an eigenvalue 1 with multiplicity $n$ and the remaining eigenvalues are  the eigenvalues of $T=\left( \beta I + G\right)^{-1}G$ and this completes the proof.
\end{proof}

\begin{remark}
 Similar to Theorem 3.2 in \cite{SSP} it can be shown that the dimension of the Krylov subspace $\mathcal{K}(\Psi,b)$ is at most $m+1$. This shows that the GMRES iterative method to solve \eqref{sdpr1} in conjunction with the preconditioner $\mathcal{P}_{RMGSS}^{-1}$ terminates in most $m+1$ iterations and provides the exact solution of the system (see Proposition 6.2 in \cite{Saadbook}). Obviously, the matrix $G$ is SPD and as a result its eigenvalues are positive. Therefore, from Theorem \ref{th2} we see that $|\lambda_i-1|=0$ or $|\lambda_i-1|=\frac{\beta}{\beta+\mu_i}$. Hence, the eigenvalues of $\Psi$ would be well clustered with a nice clustering of its eigenvalues around the point $(1,0)$ for small values of $\beta$. In this case, the matrix $\Psi$ would be well conditioned.
 \end{remark}

\section{Numerical Experiments} \label{Sec3}

In this section, some numerical experiments are given to show the effectiveness of the   MGSS and RMGSS preconditioners.
All the numerical experiments presented in this section were computed in double precision using some MATLAB codes on a Laptop with Intel Core i7 CPU 1.8 GHz, 6GB RAM. We consider the Stokes problem (see \cite[page 221]{Elman})
 \begin{equation}\label{ExSDP}
\left\{\begin{array}{ll}
-\triangle \textbf{u}+\nabla p=\textbf{f}, \\
\hspace{1.1cm}\nabla . \textbf{u}=0, \\
\end{array}\right.
\end{equation}
in $\Omega=[-1,1] \times [-1,1]$, with the exact solution
\[
\textbf{u}=(20xy^3,5x^4-5y^4),\quad p=60x^2y-20y^3+\textrm{constant}.
\]
We use the interpolant of $\textbf{u}$ for specifying Dirichlet conditions everywhere on the boundary. The test problems  were generated by using the IFISS software package written by Elman et al. \cite{prec1}. The IFISS package were used to discretize the problem \eqref{ExSDP} using stabilized Q1-P0 finite elements. We used $\beta=0.25$ as the stabilization parameter. Matrix properties of the test problem for different sizes are given in Table \ref{Tbla}.

We use GMRES($5$) in conjunction with the preconditioners $\mathcal{P}_{MGSS}$ and $\mathcal{P}_{RMGSS}$. We also compare the results of the MGSS and RMGSS preconditioners with those of the Hermitian and skew-Hermitian (HSS) preconditioner (see \cite{Benzi-a1}). To show the effectiveness of the methods we also give the results of  GMRES(5) without preconditioning.  We use a null vector as an initial guess and the  stopping criterion $ \|b-Ax^{(k)}\|_2< 10^{-9}\| b\|_2$. In the implementation  of the preconditioners $\mathcal{P}_{MGSS}$ and $\mathcal{P}_{RMGSS}$, in Algorithm 1,  we use the Cholesky factorization of $\beta I+C$ and the CG method to solve the system of Step 3. It is noted that, in the CG method, the iteration is terminated when the residual norm is reduced by a factor of 100 or when the number of iterations exceeds 40.  Numerical results are given in Table \ref{Tbl} for different sizes of the problem. In this table ``IT" denotes for the number of iterations for the convergence and ``CPU" stands for the corresponding CPU times (in seconds). For the HSS preconditioner we experimentally computed the optimal value of the involving parameter (see \cite{Benzi-a1}) of the method.  For the MGSS method we present the numerical results for $(\alpha,\beta)=(0.01,0.001)$ and $(\alpha,\beta)=(0.001,0.001)$  and in the RMGSS method for  $\beta=0.001$. As the numerical results show the all the preconditioners are effective. We also observe that the MGSS and RMGSS preconditioners are superior to the HSS preconditioners in terms of both iteration count and CPU times. For more  investigation the eigenvalues distribution of the matrices $\mathcal{A}$, $\mathcal{P}_{MGSS}^{-1}\mathcal{A}$ with $\alpha=\beta=0.001$ and  $\mathcal{P}_{RMGSS}^{-1}\mathcal{A}$ with $\beta=0.001$  are displayed in Figure 1. As we see the eigenvalues of  $\mathcal{P}_{MGSS}^{-1}\mathcal{A}$ and  $\mathcal{P}_{RMGSS}^{-1}\mathcal{A}$ are more clustered than the matrix $\mathcal{A}$.

\section{Conclusion}\label{Sec4}

We have presented a modification of the generalized shift-splitting method to solve the saddle point problem with symmetric positive definite (1,1)-block and symmetric positive semidefinite (2,2)-block. Then the resulted preconditioner and its relaxed version have been implemented to precondition the saddle point problem. We have seen that both of the preconditioners are effective when they are combined with the GMRES(m) algorithm. Our numerical results show that the proposed preconditioners are more effective than the HSS preconditioner.

\begin{table}
 \caption{Matrix properties of the test problem. \label{Tbla}}
 \centering
\begin{tabular}{ccccccccccccccc}\\ \hline    
Grid             & $n$ & $m$ & nnz(A) &  nnz(B) & nnz(C) \\ \hline
$16 \times 16$   & 578   & 256  & 3826  & 1800  & 768\\[1mm]
$32 \times 32$   & 2178  & 1024 & 16818 & 7688  & 3072\\[1mm]
$64 \times 64$   & 8450  & 4096 & 70450 & 31752 & 12288\\[1mm]
$128 \times 128$ & 33282 & 16384& 288306 & 129032 & 49152\\ \hline
\end{tabular}
\label{Tbl}
\end{table}

\begin{table}
 \caption{Numerical results for the test problem. \label{Table2}}
{\scriptsize
\begin{tabular}{ccccccccccccccc}\\ \hline \\   
               & GMRES(5)\hspace{-1cm}  & & & & MGSS &   & && RMGSS & & && HSS &  \\ \cline{2-3} \cline{5-7} \cline{9-11}  \cline{13-15} \\[0mm]
Grid           &  IT   &  CPU     & &($\alpha$,$\beta$) & IT & CPU && $\beta$ & IT & CPU & & $\alpha$ & IT & CPU  \vspace{-0.0cm} \\ \hline \\[1mm]
$16 \times 16$ &18 & 0.119  & & $(0.01,0.001)$  & 6  & 0.048    &&  0.001 & 6  & 0.048  && 0.085  &  12 & 0.051\\[1mm]
               &   &      & &  $(0.001,0.001)$ & 6  & 0.048    &&        &    &        &&        &     &      \\[2mm]

$32 \times 32$ &33 & 0.530      & &  $(0.01,0.001)$  & 6  & 0.152    &&  0.001 & 5  & 0.145  && 0.050  &  18 & 0.199\\[1mm]
               &   &      & &  $(0.001,0.001)$ & 6  & 0.153    &&        &    &        &&        &     &      \\[2mm]

$64 \times 64$ &147 & 8.066      & &  $(0.01,0.001)$  & 14 & 1.305    &&  0.001 & 7  & 1.021  && 0.020  &  27 & 2.449\\[1mm]
               &    &      & &  $(0.001,0.001)$ & 7  & 0.905    &&        &    &        &&        &     &      \\[2mm]

$128\times 128$&349 & 76.9  & &  $(0.01,0.001)$  & 27 & 10.492   &&  0.001 & 15 & 10.016 && 0.020  &  41 & 24.633\\[1mm]
               &    &     & &  $(0.001,0.001)$ & 14 & 9.284    &&        &    &        &&        &     &      \\ \hline
\end{tabular}}
\label{Tbl}
\end{table}

\begin{figure}
\centerline{\includegraphics[height=5cm,width=5.5cm]{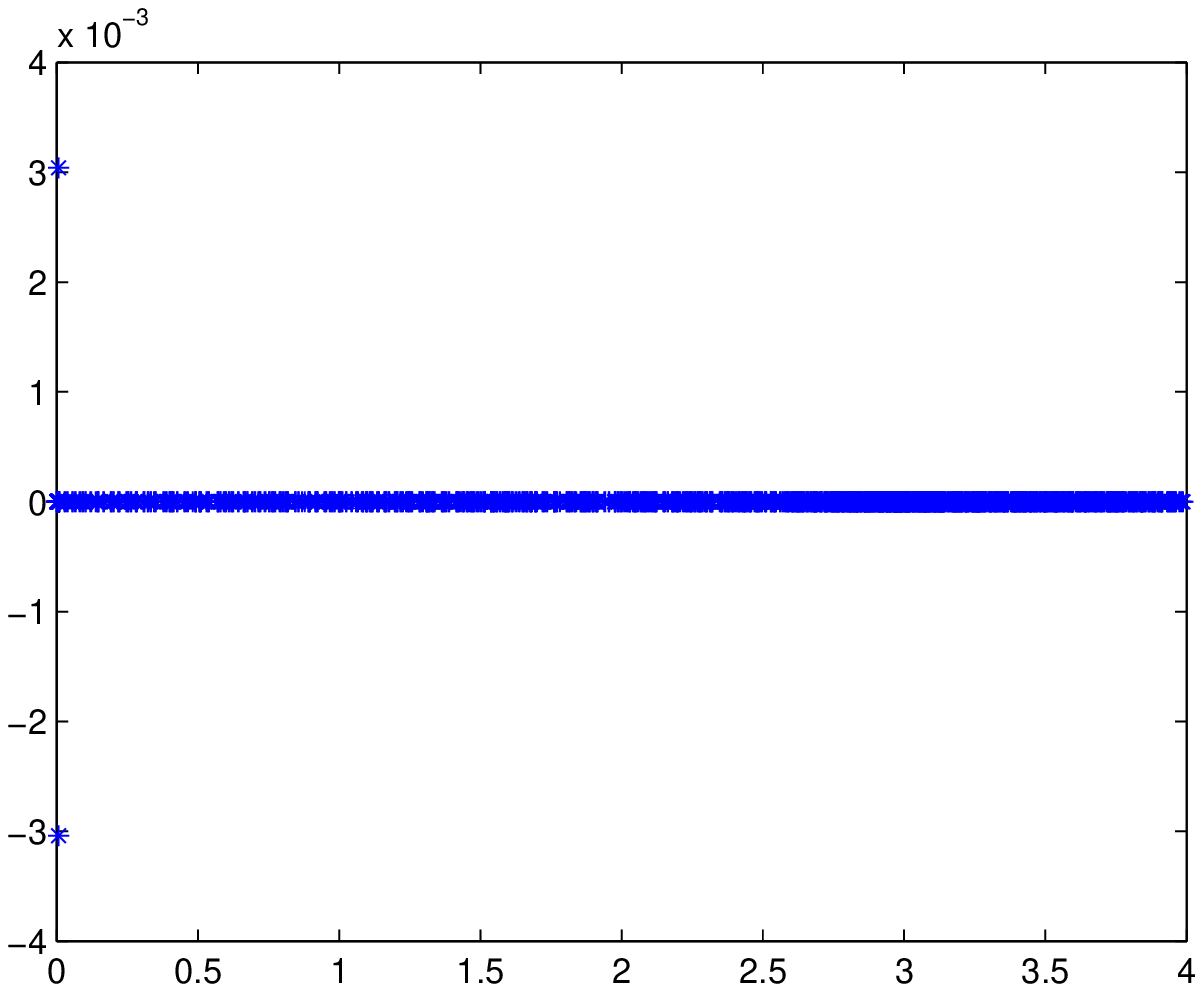}\includegraphics[height=5cm,width=5.5cm]{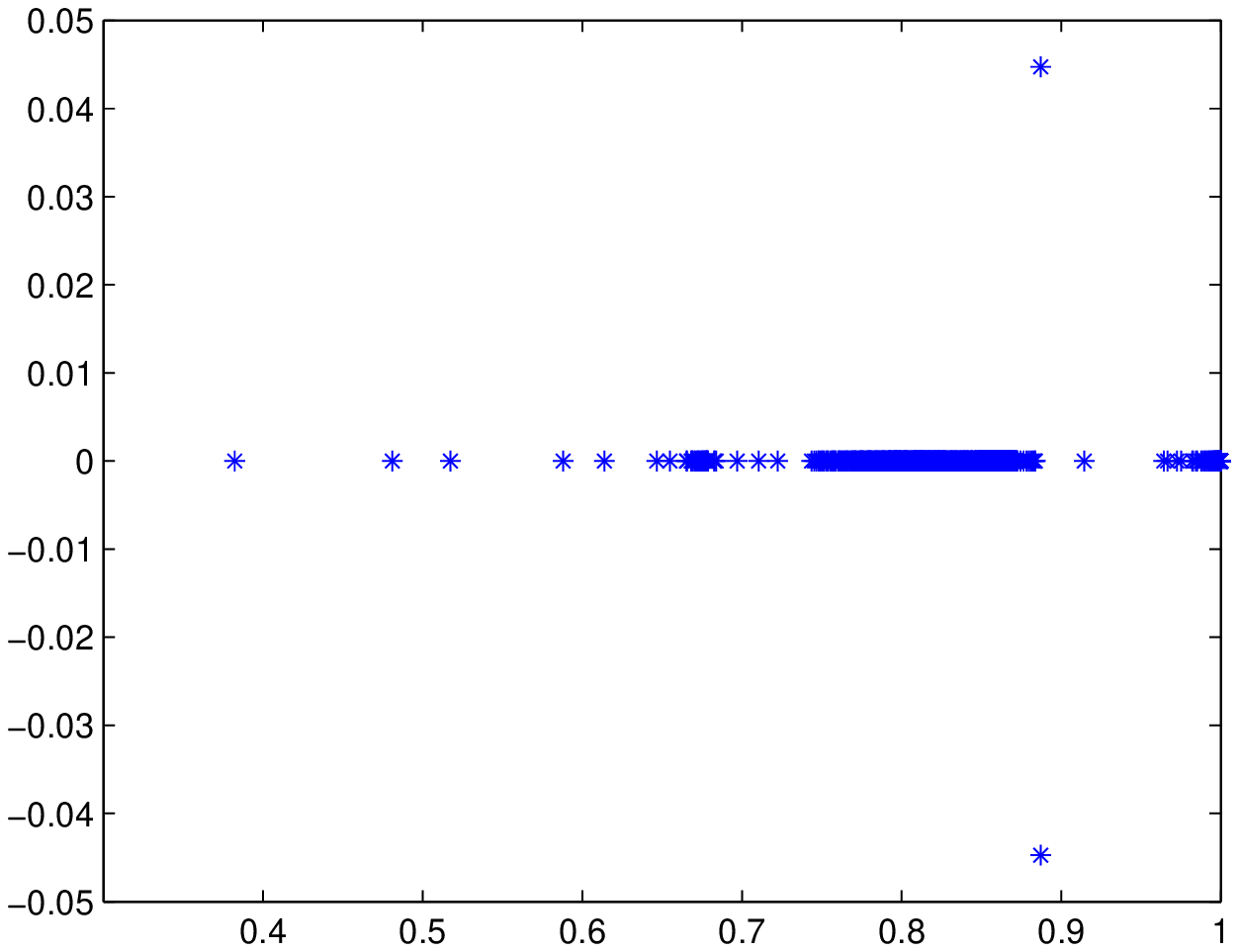}\includegraphics[height=5cm,width=5.5cm]{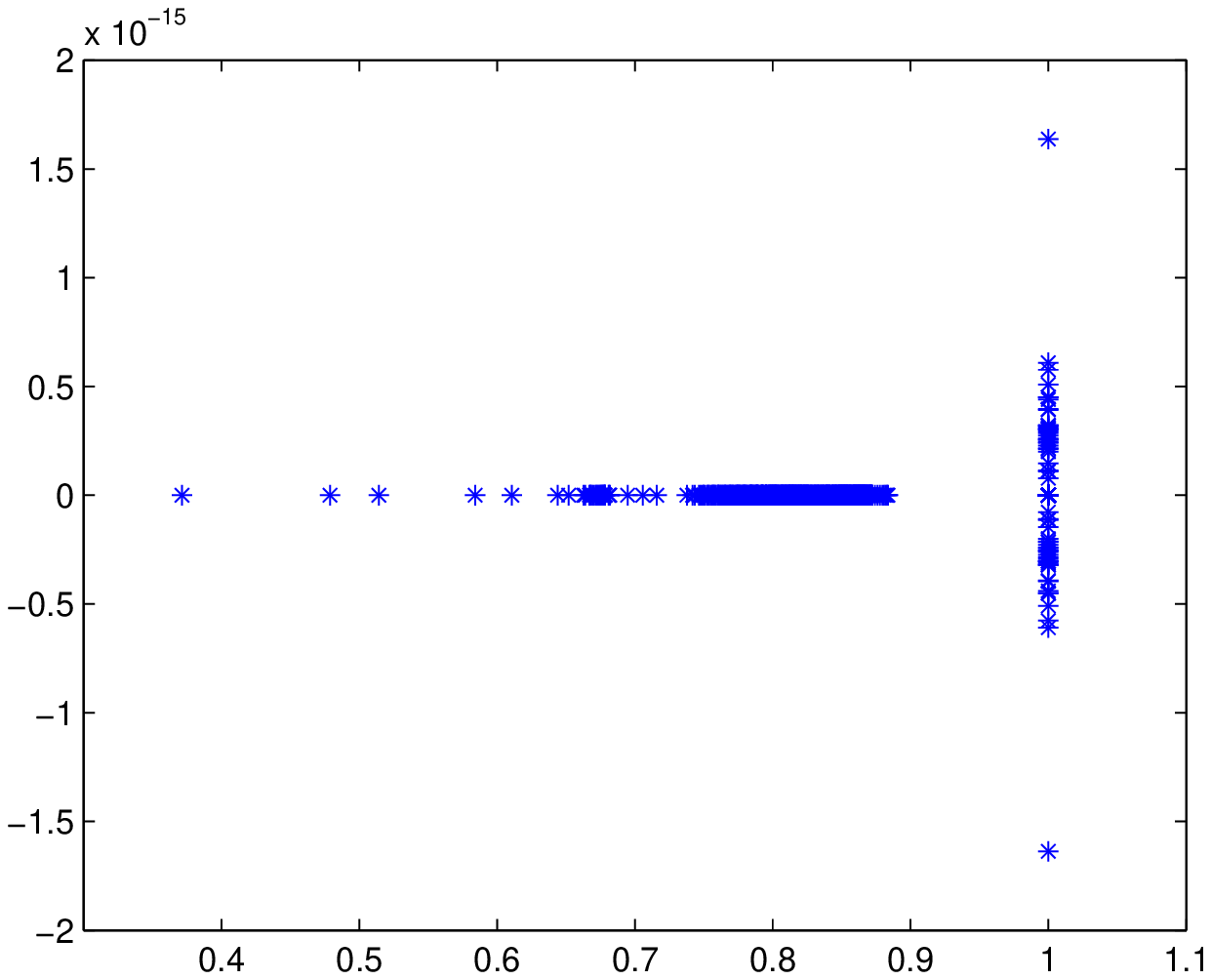}}
{\caption{Eigenvalues distribution of  the saddle point matrix $\mathcal{A}$ (left) and the preconditioned matrix, $\mathcal{P}_{MGSS}^{-1}\mathcal{A}$ where $\alpha=\beta=0.001$ (middle) and $\mathcal{P}_{RMGSS}^{-1}\mathcal{A}$ where $\beta=0.001$ (right)  with $m=32$.}}
\label{Fig}
\end{figure}

 \section*{Acknowledgements}
The authors are grateful to the anonymous referees  for their valuable
comments and suggestions which improved the quality of this paper.

\enddocument